\documentclass[11pt]{amsart}

\usepackage[utf8]{inputenc}  

\usepackage{amsthm}
\usepackage{amsfonts}
\usepackage{amsmath}
\usepackage{amssymb}
\usepackage{mathtools} 

\usepackage{textcomp} 
\usepackage{xcolor}
\definecolor{darkgreen}{rgb}{0,0.45,0}
\definecolor{darkred}{rgb}{0.75,0,0}
\definecolor{darkblue}{rgb}{0,0,0.6}
\definecolor{dkblue}{rgb}{0,0.1,0.5}
\definecolor{lightblue}{rgb}{0,0.5,0.5}
\definecolor{dkgreen}{rgb}{0,0.4,0}
\definecolor{dk2green}{rgb}{0.4,0,0}
\definecolor{dkviolet}{rgb}{0.6,0,0.8}
\usepackage[colorlinks,citecolor=darkgreen,linkcolor=darkred,urlcolor=darkblue]{hyperref}
\usepackage{breakurl}
\usepackage{listings}

\usepackage{fancyvrb}

\usepackage{enumerate} 

\usepackage{tikz}
\usetikzlibrary{arrows}
\usepackage[all]{xy}
\xyoption{2cell}
\xyoption{curve}
\UseTwocells
\input{diagxy}  
\usetikzlibrary{matrix,arrows}

\pgfarrowsdeclare{xyto}{xyto}
{
  \arrowsize=0.22pt
  \advance\arrowsize by .7\pgflinewidth  
  \pgfarrowsleftextend{-10\arrowsize-.5\pgflinewidth}
  \pgfarrowsrightextend{.5\pgflinewidth}
}
{
  \arrowsize=0.22pt
  \advance\arrowsize by .7\pgflinewidth  
  \pgfsetdash{}{0pt} 
  \pgfsetroundjoin	
  \pgfsetroundcap	
  \pgfpathmoveto{\pgfpointorigin}
  \pgfpathcurveto
     {\pgfpoint{-4.2\arrowsize}{0.6\arrowsize}}
     {\pgfpoint{-7.2\arrowsize}{1.5\arrowsize}}
     {\pgfpoint{-10\arrowsize}{3.1\arrowsize}}
  \pgfusepathqstroke
  \pgfpathmoveto{\pgfpointorigin}
  \pgfpathcurveto
     {\pgfpoint{-4.2\arrowsize}{-0.6\arrowsize}}
     {\pgfpoint{-7.2\arrowsize}{-1.5\arrowsize}}
     {\pgfpoint{-10\arrowsize}{-3.1\arrowsize}}
  \pgfusepathqstroke
}

\newlength{\ontotipoffset}

\pgfarrowsdeclare{xyonto}{xyonto}
{
  \arrowsize=0.22pt
  \advance\arrowsize by .7\pgflinewidth  
  \pgfarrowsleftextend{-10\arrowsize-.5\pgflinewidth}
  \pgfarrowsrightextend{.5\pgflinewidth}
}
{
  \arrowsize=0.22pt
  \advance\arrowsize by .7\pgflinewidth  
  \ontotipoffset=-6\arrowsize
  \pgfsetdash{}{0pt} 
  \pgfsetroundjoin	
  \pgfsetroundcap	
  \pgfpathmoveto{\pgfpointorigin}
  \pgfpathcurveto
     {\pgfpoint{-4.2\arrowsize}{0.6\arrowsize}}
     {\pgfpoint{-7.2\arrowsize}{1.5\arrowsize}}
     {\pgfpoint{-10\arrowsize}{3.1\arrowsize}}
  \pgfusepathqstroke
  \pgfpathmoveto{\pgfpointorigin}
  \pgfpathcurveto
     {\pgfpoint{-4.2\arrowsize}{-0.6\arrowsize}}
     {\pgfpoint{-7.2\arrowsize}{-1.5\arrowsize}}
     {\pgfpoint{-10\arrowsize}{-3.1\arrowsize}}
  \pgfusepathqstroke
  \pgfpathmoveto{\pgfpoint{\ontotipoffset}{0}}
  \pgfpathcurveto
     {\pgfpoint{\ontotipoffset-4.2\arrowsize}{0.6\arrowsize}}
     {\pgfpoint{\ontotipoffset-7.2\arrowsize}{1.5\arrowsize}}
     {\pgfpoint{\ontotipoffset-10\arrowsize}{3.1\arrowsize}}
  \pgfusepathqstroke
  \pgfpathmoveto{\pgfpoint{\ontotipoffset}{0}}
  \pgfpathcurveto
     {\pgfpoint{\ontotipoffset-4.2\arrowsize}{-0.6\arrowsize}}
     {\pgfpoint{\ontotipoffset-7.2\arrowsize}{-1.5\arrowsize}}
     {\pgfpoint{\ontotipoffset-10\arrowsize}{-3.1\arrowsize}}
  \pgfusepathqstroke
}

\pgfarrowsdeclare{xytri}{xytri}
{
  \arrowsize=0.22pt
  \advance\arrowsize by .7\pgflinewidth  
  \pgfarrowsleftextend{-.5\pgflinewidth}
  \pgfarrowsrightextend{10\arrowsize+.5\pgflinewidth}
}
{
  \arrowsize=0.22pt
  \advance\arrowsize by .7\pgflinewidth  
  \pgfsetdash{}{0pt} 
  \pgfsetroundjoin	
  \pgfsetroundcap	
  \pgfpathmoveto{\pgfpoint{10\arrowsize}{0}}
  \pgfpathcurveto
     {\pgfpoint{5.8\arrowsize}{0.6\arrowsize}}
     {\pgfpoint{2.8\arrowsize}{1.5\arrowsize}}
     {\pgfpoint{0\arrowsize}{3.1\arrowsize}}
  \pgfpathmoveto{\pgfpoint{10\arrowsize}{0}}
  \pgfpathcurveto
     {\pgfpoint{5.8\arrowsize}{-0.6\arrowsize}}
     {\pgfpoint{2.8\arrowsize}{-1.5\arrowsize}}
     {\pgfpoint{0\arrowsize}{-3.1\arrowsize}}
  \pgflineto{\pgfpoint{0}{3.1\arrowsize}}
  \pgfusepathqstroke
} 
\tikzset{>=xyto}
\pgfarrowsdeclarealias{c}{c}{right hook}{left hook}
\pgfarrowsdeclarealias{c'}{c'}{left hook}{right hook}

\tikzset{
  cd/.style={->,auto,font=\scriptsize},
  homot/.style={cd,double,double equal sign distance,-implies,shorten >=0.5ex,shorten <=0.5ex},
  close/.style={inner sep=2.5pt},
  closer/.style={inner sep=1.5pt},
  closest/.style={inner sep=0.5pt},
  descr/.style={fill=white}}

\newbox\pbbox
\setbox\pbbox=\hbox{\xy \POS(65,0)\ar@{-} (0,0) \ar@{-} (65,65)\endxy}
\def\pb{\save[]+<3.5mm,-3.5mm>*{\copy\pbbox} \restore}
\theoremstyle{plain}
\newtheorem{theorem}{Theorem}[subsection]

\newtheorem{fact}[theorem]{Fact}
\newtheorem{proposition}[theorem]{Proposition}
\newtheorem{lemma}[theorem]{Lemma}
\newtheorem{corollary}[theorem]{Corollary}

\theoremstyle{definition}
\newtheorem{definition}[theorem]{Definition}

\newtheorem{example}[theorem]{Example}

\newtheorem{remark}[theorem]{Remark}

\makeatletter
\renewcommand{\paragraph}{\@startsection{paragraph}{4}{0mm}{-0.5\baselineskip}{-1ex}{\bf}}
\makeatother

\makeatletter
\ifcsname phantomsection\endcsname
    \newcommand*{\qrr@gobblenexttocentry}[5]{}
\else
    \newcommand*{\qrr@gobblenexttocentry}[4]{}
\fi
\newcommand*{\addsubsection}{%
    \addtocontents{toc}{\protect\qrr@gobblenexttocentry}%
    \subsection}
\makeatother

\newcommand{\one}{\mathbf{1}}
\newcommand{\C}{\mathbf{C}}
\newcommand{\SynCat}[1]{\mathcal{C}({#1})}
\newcommand{\cF}{\mathcal{F}}
\newcommand{\I}{\mathcal{I}}
\newcommand{\N}{\mathbb{N}}
\renewcommand{\P}{\mathbf{\mathsf{P}}}

\newcommand{\T}{\mathbf{T}}
\newcommand{\cW}{\mathcal{W}}
\newcommand{\Z}{\mathbb{Z}}

\newcommand{\pinv}[1]{\overline{#1}}

\newcommand{\Ho}{\mathrm{Ho}}
\newcommand{\inv}[1]{ #1 ^{-1}}
\newcommand{\pt}{\mathsf{pt}}
\newcommand{\FML}{\ensuremath{\mathcal{H}'}}
\newcommand{\Top}{\mathbf{Top}}

\newcommand{\fib}{\twoheadrightarrow}

\newcommand{\iso}{\cong}
\newcommand{\homot}{\Rightarrow}
\renewcommand{\equiv}{\simeq}
\newcommand{\pdot}{\centerdot}

\newcommand{\name}[1]{{\ulcorner {#1} \urcorner}}

\newcommand{\of}{\mathord{:}}
\newcommand{\types}{\vdash}

\newcommand{\synId}{\mathsf{Id}}
\newcommand{\synNat}{\mathsf{Nat}}
\newcommand{\synBool}{\mathsf{Bool}}

\newcommand{\synOne}{\mathsf{1}}

\newcommand{\synU}{\mathsf{U}}

\newcommand{\refl}{\mathsf{refl}}

\let\syn\mathsf
\newcommand{\Cone}{\syn{Cone}}

\newcommand{\hfib}{\syn{hfib}}
\newcommand{\isContr}{\syn{isContr}}
\newcommand{\Eq}{\syn{Eq}}
\newcommand{\Pb}{\syn{Pb}}
\newcommand{\Lim}{\syn{Lim}}
\newcommand{\Id}[3][]{{(#2 \leadsto_{#1} #3)}}

\newenvironment{internal}{\fontfamily{\sfdefault}\selectfont}{}

\newcommand{\todo}[1]{}

\begin{document}
\title{Homotopy limits in type theory}

\author{Jeremy Avigad}
\address[Jeremy Avigad]{Carnegie Mellon University, Pittsburgh}
\email{avigad@cmu.edu}

\author{Krzysztof Kapulkin}
\address[Krzysztof Kapulkin]{University of Pittsburgh}
\email{krk56@pitt.edu}

\author{Peter LeFanu Lumsdaine}
\address[Peter LeFanu Lumsdaine]{Institute for Advanced Study, Princeton}
\email{plumsdaine@ias.edu}

\date{March 1, 2013}

\begin{abstract}
Working in homotopy type theory, we provide a systematic study of homotopy limits of diagrams over graphs, formalized in the Coq proof assistant.  We discuss some of the challenges posed by this approach to formalizing homotopy-theoretic material. We also compare our constructions with the more classical approach to homotopy limits via fibration categories.
\end{abstract}

\maketitle

\tableofcontents

\section{Introduction}

Homotopy type theory is based on the discovery that formal dependent type theory has a natural homotopy-theoretic interpretation (\cite{voevodsky:homotopy-lambda-calculus}, \cite{awodey-warren}).  Since a number of interactive proof assistants implement versions of dependent type theory, these observations open the possibility of developing parts of homotopy theory formally with the help of such assistants; see \cite{pelayo-warren:univalent-foundations-paper} for a helpful overview.

In this spirit, we carry out a number of homotopy-theoretic constructions in a core system of homotopy type theory. In particular, we define and investigate (homotopy) pullbacks, equalizers, limits over graphs, pointed spaces, and fiber sequences. The entire development is formalized with the Coq interactive proof assistant.  Besides the formalization itself, we also compare the semantics of type theory with fibration categories, a standard homotopy-theoretic setting for the construction of homotopy limits.

We assume some familiarity with type theory, but not specifically with the homotopical version; we do not assume any previous acquaintance with homotopy limits.

We should mention that many of the facts we present below are already known in folklore; in any case, none of them will be unexpected to researchers in the field.  Egbert Rijke and Bas Spitters \cite{rijke-spitters:lims-and-colims-in-hott} have also independently investigated limits and colimits over graphs within a similar type theory.  We hope it will prove useful, however, to have a systematic treatment of these basic results, fully formalized in Coq and available as a library for future use. We also hope that the practical lessons we learned during the formalization process may be useful to others.

Our Coq development builds on a library for homotopy type theory developed jointly by various people, under the leadership of Andrej Bauer, Lumsdaine, and Michael Shulman \cite{hott:repo}. Another extensive library has been developed by Vladimir Voevodsky \cite{voevodsky:univalent-foundations-coq}, and some of our verified results overlap his.

\addsubsection*{Outline}
In Section~\ref{framework:section}, we set out the the formal framework of our work: the type theory under consideration, and its intended interpretation. Along with this, we very briefly review homotopy limits in the classical setting.  In Section~\ref{sec:fundamentals}, we recall some key constructions from the type-theoretic development of homotopy theory, and use these to show that every categorical model of the theory carries the structure of a fibration category. Section~\ref{sec:further-development} presents the main body of our formalization: a concise treatment of the content, in traditional mathematical prose. Finally, in Section~\ref{reflections:section} we share some reflections on practical aspects of the formalization process.

Our formal development in Coq can be found in the files associated with the journal publication of this paper, and also online at
\begin{center}
\url{https://github.com/peterlefanulumsdaine/hott-limits/tree/v1}.
\end{center}
The Github version will be maintained for compatibility with Coq and the HoTT library.
 
References to the formal code are typeset in \texttt{a teletype font}. For brevity, we omit the \texttt{.v} extension from filenames, so that, for example, \lstinline{Fundamentals.v} is cited as \lstinline{Fundamentals}. Often a single lemma in the informal presentation below translates to a cluster of formal lemmas in our files, in which case we simply cite a representative element of that cluster.

\addsubsection*{Acknowledgements.} Most of this work was carried out during the Special Year on Univalent Foundations at the Institute for Advanced Study, and all three authors are grateful to the Institute for its hospitality and support. We are also grateful to two anonymous referees for numerous helpful comments, corrections, and suggestions.

Avigad's work has been partially supported by NSF grant DMS-1068829 and AFOSR grant FA9550-12-1-0370. Kapulkin was supported by NSF Grant DMS-1001191 (P.I.~Steve Awodey) and a grant from the Benter Foundation (P.I.~Thomas Hales). Lumsdaine was supported by NSF grant DMS-1128155. Any opinions, findings, and conclusions or recommendations expressed in this material are those of the authors and do not necessarily reflect the views of the National Science Foundation.

Kapulkin dedicates this work to his mother.

\section{Background}
\label{framework:section}

In this section, we first lay out the specific logical system in which we will work.  We then review its intended semantics, insofar as they are relevant to working within the theory, and fix the basic notation and terminology we will use, based on the intended semantics.  Finally, we very briefly review classical homotopy limits.

\subsection{Logical setting}
\label{type:theory:section}

We assume familiarity with some form of dependent type theory.  Specifically, we will work in the system of \emph{predicative Martin-L\"of type theory} \cite{martin-lof:bibliopolis}. The following types, and associated rules, form a minimal core to that system:
\begin{enumerate}
 \item\label{enum:Pi} dependent products $\Pi_{x : A} B$, and the associated introduction, elimination, and computation rules
 \item\label{enum:Sigma} dependent sums $\Sigma_{x : A} B$, and the associated introduction, elimination, and computation rules
 \item\label{enum:Id} identity types $\synId_A$, and the associated introduction, elimination, and computation rules
\end{enumerate}

Most developments in homotopy type theory add at least the following rule:
\begin{enumerate}
 \setcounter{enumi}{3}
 \item\label{enum:funext} function extensionality: for any type $A$, any type $B$ depending on $x : A$, and functions $f, g :\Pi_{x : A} B$, if $f x = g x$ for every $x : A$, then $f = g$.
\end{enumerate}

The system based on these rules is used in \cite{awodey-gambino-sojakova:inductive}, where it is denoted by $\mathcal{H}$; it also forms a sufficient basis for much of the present formalization.  Some of our constructions, in addition, depend on:
\begin{enumerate}
 \setcounter{enumi}{4}
 \item the type $\synNat$ of natural numbers, with the usual introduction, elimination, and computation rules,
\end{enumerate}
from which the empty type, unit type, and other finite types can be defined; and finally, some definitions presuppose the existence of:
\begin{enumerate}
 \setcounter{enumi}{5}
 \item a universe $\synU$ of types, containing $\synNat$, and closed under the formation of dependent products, sums, and identity types.
\end{enumerate}

We use quantification over the universe to define the universal properties of pullbacks and limits, but also give equivalent formulations that do not make use of such a universe.

Besides these, the version of Coq we used implements $\eta$-conversion for functions, $\lambda x. f x = f$, as a built-in conversion rule. As a propositional equality, it is derivable from function extensionality, so we do not believe its use is essential; however, since it is unavoidably present in the proof assistant, we include it in our formal theory.

In sum, if we take axioms (\ref{enum:Pi})--(\ref{enum:Id}) to represent the core of Martin-L\"of type theory, $\mathsf{ML}$, it is then reasonable to denote our overall framework as
\[
 \mathsf{ML} + (\mathsf{funext}) + (\eta) + (\synNat) + (\synU).
\]
For brevity, we refer to this in the present paper as \FML; thus the formal content of our work is that the constructions and assertions of Sections~\ref{sec:fundamentals} and~\ref{sec:further-development} are consequences of this formal theory. As noted above, however, most of our results do not require $\synNat$, and many do not require $\synU$.

We do not consider in the present work extra axioms such as Univalence, resizing, or higher inductive types.

One final comment about the formal verification: rather than providing $\synId$, $\synNat$, $\synBool$, and so on individually, Coq provides a general mechanism for defining inductive types, which these are then defined as instances of.  However, the resulting eliminators for these types correspond precisely to the rules for them described above.  Coq also provides (dependent) record types, as syntactic sugar for certain inductive types; in some cases, using record types made type checking more efficient, and brought notational benefits as well.  As these may be routinely translated into (iterated) $\Sigma$-types, their use has no bearing on the question of derivability in \FML.

\subsection{Semantics}
\label{sec:semantics}

\subsubsection{General algebraic semantics}

The fully general semantics of dependent type theories are, from a purely algebraic point of view, well-understood.  Essentially, a model of any dependent type theory $\T$ with the same basic judgements and structural rules as \FML\ may be defined as a \emph{contextual category}---that is, a category equipped with structure sufficient to model the structural rules---along with further algebraic structure corresponding to the logical constructors and axioms of $\T$.  For the details of this definition, see \cite{streicher:semantics-book}; for brevity, we will refer to such a structure as a \emph{categorical model of $\T$}.

The justification for calling such structures models comes from the fact that the syntax of the theory forms an initial such structure:

\begin{definition}
Given any dependent type theory $\T$, the \emph{syntactic category} $\SynCat{\T}$ is given as follows:
\begin{itemize}
\item objects of $\SynCat{\T}$ are contexts $[ x_1 \of A_1,\ \ldots,\ x_n \of A_n ]$ of $\T$, up to definitional equality and renaming of free variables;
\item maps of $\SynCat{\T}$ are \emph{context morphisms} (a.k.a.\ \emph{substitutions}), again up to definitional equality and renaming of free variables.  That is, a map
\[f \colon [ x_1 \of A_1,\ \ldots,\ x_n \of A_n] \to [ y_1 \of B_1,\ \ldots,\ y_m \of B_m(y_1, \ldots, y_{m-1}) ] \]
is represented by a sequence of terms
\begin{equation*}
\begin{split}
  x_1 \of A_1,\ \ldots,\ x_n \of A_n & \types f_1 : B_1 \\
  & \vdots  \\
  x_1 \of A_1,\ \ldots,\ x_n \of A_n & \types f_m : B_m(f_1,\ \ldots,\ f_{m-1}).
\end{split}
\end{equation*}
\end{itemize}

Moreover, $\SynCat{\T}$ may naturally be given the structure of a contextual category; for each logical rule of $\T$, $\SynCat{\T}$ carries the corresponding algebraic structure.
\end{definition}

\begin{fact}[\cite{streicher:semantics-book}\protect{\footnote{Unfortunately, to our knowledge, no general form of this result exists in the literature; it is shown for certain specific type theories in \cite{streicher:semantics-book} and elsewhere, and its extension to other combinations of the standard rules (such as \FML) is well-known in folklore.}}]
$\SynCat{\T}$ is initial among categorical models of $\T$.
\end{fact}

Thus any other categorical model $\C$ has a canonical structure-preserving functor from $\SynCat{\T}$---that is, an interpretation function, interpreting the syntax of $\T$ in $\C$.

\subsubsection{Homotopical semantics}

Homotopy type theory is based on the realization (\cite{hofmann-streicher}, \cite{awodey-warren}, \cite{garner-van-den-berg:top-and-simp-models}, \cite{voevodsky:notes-on-type-systems}) that various homotopy-theoretic settings give natural examples of such categorical models.  Very roughly, a type $A$ denotes a space; a family $B(x)$ of types, depending on some variable $x$ of type $A$, denotes a fibration over $A$; a term $t(x)$ of type $A'$, again dependent on a variable $x : A$, denotes a continuous map from $A$ to $A'$; a term $t(x)$ of type $B(x)$, dependent on $x : A$, denotes a section of the corresponding fibration over $A$; and so on.

The main motivating interpretation, for us, is the model in simplicial sets---one of the most well-studied models of spaces in homotopy theory.  The full details of this interpretation are rather technical, so since we never require them, we omit them here; see \cite{kapulkin-lumsdaine-voevodsky:simplicial-model} for a complete presentation of the simplicial set model, and \cite{shulman:inverse-diagrams} for more general related models. We sketch here just the main ingredients of the interpretation, insofar as they justify the intuition and terminology for working within the theory.

In this model, closed types (and, more generally, contexts) are interpreted as Kan complexes; dependent types, as Kan fibrations.  Most type formers---$\Pi$-types, $\Sigma$-types, $\synNat$, etc.---are interpreted as in the more familiar topos logic: $\Pi$-types by the right adjoint to pullback, $\Sigma$-types by the left, $\synNat$ by the natural numbers object, and so on.

The main novelty, however, is the interpretation of the identity type $\synId_A(x, y)$ with variables $x$ and $y$ from $A$.  In set- and topos-theoretic models, one would interpret it as the diagonal map $A \to A \times A$.  However, in simplicial sets (and other homotopy-theoretic settings) this map is hardly ever a fibration. It can, however, be replaced by a fibration $P(A) \to A \times A$, where $P(A)$ is the \emph{path object} of $A$; this is then used to interpret the identity type of $A$. Thinking of a simplicial set as a space, $P(A)$ represents the space of paths in $A$, with the fibration $P(A) \to A \times A$ giving the indexing of paths over their endpoints.  In particular contrast to the set-theoretic situation, for given $x,y : A$ the space $P(A)(x,y)$ of paths from $x$ to $y$ may be not a mere proposition, but a non-trivial space in its own right.

\subsection{Notation and terminology}

Our choices of notation and terminology are guided by the homotopical interpretation.  In particular, we will write $p : \Id{x}{y}$ for the identity type, to emphasize that we consider it as the type of paths from $x$ to $y$. The Homotopy Type Theory library uses the notation \lstinline!x = y! for this type, and in our Coq development, we stick with this. However, in the informal presentation below, we find it most natural to understand our constructions as constructions of paths, rather than equality proofs; and so we settle on the latter notation, and favor the word ``path'' over ``equality.''

In other respects, however, we have found it more convenient to leave the homotopy-theoretic interpretation implicit. For example, the natural definitions of pullbacks, equalizers, and limits in type-theoretic notation turn out to characterize homotopy pullbacks, homotopy equalizers, and homotopy limits in the homotopy-theoretic interpretation. Having kept the notion of ``path'' prominent, sprinkling the word ``homotopy'' everywhere seemed to impose an unnecessary burden; thus, both in code and in prose we refer just to ``pullbacks,'' ``equalizers,'' and ``limits.''  (This is customary in higher category theory (see, e.g.,\ \cite{lurie:htt}), when one uses, for example, the word ``limit'' for an object that in strict terms is only a homotopy limit.)

For the sake of readability, we will use standard mathematical terminology and notation in Sections~\ref{sec:fundamentals} and~\ref{sec:further-development}, rather than attempting to adhere closely to the notation used in the Coq code.  Table~\ref{table:notations} lists some of the basic notions of our development, comparing the notations used in our presentation here with those used in the Coq formalization.

\begin{table}
\begin{tabular}[b]{| c | c | c |}
 \hline
 informal& mathematical & Coq \\
 notion & notation & notation \\
 \hline
 \hline
 $p$ is a path from $x$ to $y$ & $p : \Id{x}{y}$ & \lstinline!p : x = y! \\
 \hline
 identity path at $x$ & $\refl(x)$ & \lstinline!idpath x! \\
 \hline
 concatenation of $p$ and $q$ & $p \pdot q$ & \lstinline!p @ q! \\
 \hline
 inverse of $p$ & $\pinv{p}$ & \lstinline$!p$ \\
 \hline
 $B$ is a fibration over $A$ & $B \fib A$ & \lstinline!B : A -> Type! \\
 \hline
 total space of $B$ over $A$ & $\Sigma_{x : A} B(x)$ & \lstinline!{ x : A & B x }! \\
 \hline
 dependent product of $B$ over $A$ & $\Pi_{x : A} B(x)$ & \lstinline!forall x : A, B x! \\
 \hline
 $e$ is an equivalence from $A$ to $B$ & $e \colon A \equiv B$ & \lstinline!e : A <~> B! \\
 \hline
 inverse of $e$ & $e^{-1}$ & \lstinline!e^-1! \\
 \hline
 a universe of small types & $\synU$ & \lstinline!UU! \\
 \hline
 the natural numbers & $\mathsf{Nat}$ & \lstinline!nat! \\
 \hline
\end{tabular}

\caption{Correspondence of notations}
\label{table:notations}
\end{table}

As usual in homotopy type theory, we represent logic using \emph{propositions-as-types}, with implication, conjunction, and universal and existential quantification interpreted in terms of function, product, $\Pi$-, and $\Sigma$-types respectively. Thus, for example, the functional extensionality axiom (Axiom~\ref{enum:funext} in Section~\ref{type:theory:section} above), is formally a constant of type:
\[
 \mathsf{funext} : \prod_{\substack{A : \mathsf{Type} \\ B : A \rightarrow \mathsf{Type}}} \, \prod_{f, g : \prod\limits_{x : A} B(x)} \,
  ( \prod_{x : A} \Id{f x}{g x} ) \rightarrow \Id{f}{g}.
\]

Notice that $\Sigma$-types provide a useful way of ``packaging'' related pieces of data into a single type: to illustrate this, consider Definition~\ref{def:cospan-cone} below. Formally, a \emph{cospan} consists of types $A$, $B$, and $C$, and maps $f \colon A \to C$, $g \colon B \to C$.  Given a type $X$, a \emph{cone} over this cospan with vertex $X$ consists of maps $h \colon X \to A$ and $k \colon B \to C$, and a family of paths $\Id{f(h x)}{g (k x)}$ for each $x$ in $X$. In other words, such a cone is an element of the type
\[
 \sum_{\substack{h : X \rightarrow A \\ k : X \rightarrow B}} \, \prod_{x : X} \Id{f(h x)}{g(k x)}.
\]
Thus our formal definition in Coq reads as follows:
\begin{lstlisting}
  Definition cospan_cone {A B C : Type} (f : A -> C)
    (g : B -> C) (X : Type)
  := { h : (X -> A) & { k : (X -> B)
     & forall x, paths (f(h x)) (g(k x)) }}.
\end{lstlisting}
The curly braces around the arguments \lstinline!A!, \lstinline!B!, and \lstinline!C! indicate that these are treated as \emph{implicit arguments}. This means that the user may write just \lstinline!cospan_cone f g X!, leaving the system to infer \lstinline!A!, \lstinline!B!, and \lstinline!C! from the types of \lstinline!f! and \lstinline!g!. Sometimes one needs to turn this feature off, and specify such arguments; writing \lstinline!@cospan_cone A B C f g X! tells Coq to expect all the arguments of \lstinline!cospan_cone! to be given explicitly.

\subsection{Classical homotopy limits}
\label{subsec:classical-holimits}

For the reader unfamiliar with the classical theory of homotopy limits, we briefly survey here a few of its key points.

They may be seen as a solution to the problem that ordinary (``strict'') limits are not invariant under homotopy equivalence: for instance, the two cospans below are homotopy equivalent, but their strict pullbacks are not.

\[ \begin{tikzpicture}[scale=1.8,z={(2,0)}]
\node (A) at (0,0,0) {$\ast$};
\node (B) at (1,0,0) {$\ast$};
\node (C) at (1,1,0) {$\ast$};
\node (AxC) at (0,1,0) {$\ast$};
\node (A') at (0,0,1) {$\ast$};
\node (B') at (1,0,1) {$[0,1]$};
\node (C') at (1,1,1) {$\ast$};
\node (A'xC') at (0,1,1) {$\emptyset$};
\draw[cd] (A) to (B);
\draw[cd] (C) to (B);
\draw[cd] (AxC) to (A);
\draw[cd] (AxC) to (C);
\draw (0.1,0.75) -- (0.25,0.75) -- (0.25,0.9);
\draw[cd] (A') to node {$0$} (B');
\draw[cd] (C') to node {$1$}(B');
\draw[cd] (A'xC') to (A');
\draw[cd] (A'xC') to (C');
\draw (0.1,0.75,1) -- (0.25,0.75,1) -- (0.25,0.9,1);
\end{tikzpicture} \]
This may be resolved by instead defining the \emph{homotopy pullback} $A \times^h_B C$, as the space of triples $(a,c,p)$, where $a \in A$, $c \in C$, and $p$ is a path in $B$ from $f(a)$ to $g(c)$; the equalities in the definition of the strict pullback have been replaced by paths.

More generally, the homotopy limit of a functor $F \colon \I \to \Top$ may be defined using the \emph{end} formula $\int_{i \in \I} F(i)^{\mathbf{B}(\I/i)}$.  This has the effect of replacing equalities by homotopies, in a coherent fashion; the coherence is encoded by the use of the classifying spaces $\mathbf{B}(\I/i)$.
This generalizes to other settings, first by a similar concrete construction (in e.g.\ simplicial settings \cite{bousfield-kan:book}), and more abstractly in terms of derived functors (for general homotopical categories \cite{dwyer-hirschhorn-kan-smith}).

In the $\infty$-categorical setting, one may take an alternative approach, defining the (homotopy) limit by an $\infty$-categorical universal property directly generalizing that of ordinary 1-categorical limits (see e.g.\ \cite{lurie:htt}).   In Homotopy Type Theory, we do the same.  It turns out, in fact, that at least for diagrams over graphs, what looks like the ordinary definition of a strict set-theoretic limit actually defines the homotopy limit---both as an explicit construction, and as a characterization via a universal mapping property.

\section{Fibration categories from type theory}
\label{sec:fundamentals}

In this section and the next, we develop the basic theory of homotopy limits and related notions in \FML.  We have already explained, in Section~\ref{framework:section}, how the basic ingredients are represented in the language of Coq, and complete details of the whole development can be found in the files comprising our formal verification.  Especially in Section~\ref{sec:further-development}, therefore, we will generally only sketch most proofs, leaving out steps that are straightforward and routine (and even some that are not).

\subsection{Basic constructions}
\label{sec:hott-background}

Our formal work builds on the HoTT library \cite{hott:repo} for homotopy theory developed by Bauer, Lumsdaine, Shulman, and others.  We begin by summarizing some of the basic components of this library that are used throughout.

\subsubsection{Operations on paths}

Given any $x, y : X$, we write $p : \Id{x}{y}$ to denote that $p$ is a path from $x$ to $y$. For every $x$, there is an ``identity path'' $\refl(x) : \Id{x}{x}$. The central property characterizing the type of paths is its elimination principle, which says roughly that to construct an object of a type $C(x, y, p)$ depending on a path $p$ from $x$ to $y$, it suffices to construct an element of $C(x, x, \refl(x))$, in which $p$ has been ``contracted'' to an identity path.

Paths admit various operations familiar from homotopy theory and higher category theory. Any two paths $p : \Id{x}{y}$ and $q : \Id{y}{z}$ can be concatenated, yielding a path $p \pdot q : \Id{x}{z}$. Moreover, $\refl(x) : \Id{x}{x}$ is a unit element for this operation, and every path admits an inverse $\bar{p} : \Id{y}{x}$. These operations satisfy the groupoid laws, but, as in homotopy theory, only up to a higher path. For example, we can find an inhabitant of the type $\Id{p \pdot \bar{p}}{\refl(x)}$.  In fact, every type, together with the tower of its paths, forms an $\infty$-groupoid of some sort; precise statements along these lines can be found in \cite{garner-van-den-berg}, \cite{lumsdaine:tlca-proceedings}.

Moreover, the maps between types respect the paths and the structure on them.  That is: given any $p : \Id{x}{y}$ in $X$ and $f \colon X \to Y$, we obtain a path $f[p] : \Id{f(x)}{f(y)}$; and this is functorial, in the up-to-homotopy sense that there is, for example, an inhabitant of the type $\Id{f[p \pdot q]}{f[p] \pdot f[q]}$.

\subsubsection{Equivalences and truncatedness}
\label{sec:equiv-and-trunc}
The notion of paths allows us to recover several familiar notions from algebraic topology.

We can, for example, say that a type $X$ is \emph{contractible} if there is some $x_0 : X$, and a function giving for each $x : X$ a path $\Id{x}{x_0}$.  Formally, the proposition ``$X$ is contractible'' is defined as follows:\footnote{One might at first read this as a definition of connectedness---for each $x$, there exists some path from $x$ to $x_0$---but remember that one should think of the function sending $x$ to the path as \emph{continuous}, so as giving a contraction of $X$ to $x_0$.  Precisely, in the simplicial and similar interpretations, the $\Pi$-type becomes a space of continuous functions, and so $\isContr$ gets interpreted as the property of contractibility; and moreover, working within the theory, the logic forces $\isContr$ to behave like contractibility, not like connectedness.}
\[ \isContr(X) := \sum_{x_0:X} \prod_{x:X} \Id{x}{x_0}.  \]

One can also construct the \emph{homotopy fiber} of a map $f \colon X \to Y$ over an element $y : Y$ by:
\[ \hfib(f,y) := \sum_{x : X} \Id{f(x)}{y}. \]

Given these we say that a map $f \colon X \to Y$ is an \emph{equivalence} if for all $y : Y$ the homotopy fiber of $f$ over $y$ is contractible. The HoTT library provides many crucial results on equivalences. For example, a map is an equivalence exactly if it has a two-sided inverse (up to homotopy), or alternatively two one-sided inverses.

Another notion that smoothly transfers from algebraic topology to HoTT is the notion of an $n$-type. Classically, an $n$-type is a space whose homotopy groups vanish above dimension $n$. In HoTT we define an analogous hierarchy.

Precisely, \emph{$n$-truncatedness} is defined by induction for $n \geq -2$. A type $X$ is $(-2)$-truncated if it is contractible; and is $n+1$-truncated if for all $x, y : X$, the type $\Id{x}{y}$ of paths from $x$ to $y$ is of $n$-truncated.  For short, we refer to $n$-truncated types as \emph{$n$-types}.  In particular, $(-1)$-types may be considered as propositions, carrying no more information than the fact of being inhabited; and $0$-types as (up-to-homotopy) discrete sets.  We call such types \emph{propositions} (or \emph{mere} propositions, for emphasis) and \emph{sets} respectively.

\subsubsection{Functional extensionality}

Given two types $X$ and $Y$, the type $X \rightarrow Y$ of maps between them can be equipped with the notion of a path (or rather, a ``homotopy'') in two different ways. First, for any $f, g \colon X \to Y$, one can form $\Id{f}{g}$, in the usual way. On the hand, one can also compare two functions pointwise, asking for an element of $\prod_{x : X} \Id{f(x)}{g(x)}$; we call such a function $h$ a \emph{homotopy} from $f$ to $g$, and write $h : f \homot g$.

Given any $p : \Id{f}{g}$, we obtain by the elimination principle for paths an element of the type $f \homot g$. The functional extensionality axiom implies that this assignment is an equivalence; that is, that given a pointwise homotopy between two maps, we can always find a path between them in the function type inducing the original homotopy.  More generally, functional extensionality implies this equivalence between paths and homotopies in \emph{dependent} function types $\prod_{x : A} B(x)$.

\subsubsection{Dependent sums}

The interaction between dependent sums and paths is crucial in our work. Let $B(x)$ be a type depending on $x:A$. It is easy to see that then a path $p : \Id{a}{a'}$ in $A$ induces an equivalence $p_! \colon B(a) \to B(a')$, which we call \emph{transport} between fibers. As everything before, this commutes appropriately with the operations on paths; for example, for any $p : \Id{a}{a'}$ and $q : \Id{a'}{a''}$, and $b : B(a)$ we have $\Id{(p \pdot q)_!b}{q_!(p_! b)}$.

This also provides a means to construct paths between two elements of a $\Sigma$-type. Given a path $p : \Id{(a, b)}{(a', b')}$ in a $\sum_{x : A}B(x)$, we get a pair of paths: $p_1 : \Id{a}{a'}$ and $p_2 : \Id{(p_1)_! b}{b'}$; and conversely, given such a pair of paths, we can recover the original path $p$.  This construction is ubiquitous in the formalization, since so many objects are defined using $\Sigma$-types; for more discussion, see Section~\ref{sec:constructing-paths} below.

\subsection{Fibration category structure}
\label{fibcats:section}

In this section, we show that any categorical model of \FML\ (so, in particular, its syntactic category) satisfies the axioms of a fibration category, following the lines of results such as \cite{gambino-garner}, \cite{lumsdaine:hit-model-cat}.  After this, we look at how some standard properties of fibration categories translate in terms of the type theory.

The results follow from a combination of internal reasoning---proving certain statements in the type theory---and external (meta-theoretic), showing how in models, the internal statements translate into the desired axioms.  Since we will be switching back and forth frequently between these two different logical settings, we use \textsf{sans serif text} in this section to distinguish the \textsf{internal reasoning} from the external.  The internal portions are formalized in the file \lstinline{Fundamentals}.

We start by recalling the definition of a fibration category (for more on which, see \cite{brown:abstract-homotopy-theory}, \cite{baues:algebraic-homotopy}):

\begin{definition}
 A \emph{fibration category} is a category $\C$ together with two distinguished classes of maps, $\cW$ (the  \emph{weak equivalences}) and $\cF$ (the \emph{fibrations}) satisfying the following conditions:
 \begin{enumerate}
  \item Weak equivalences satisfy the 2-out-of-6 condition; i.e.,\ given a composable triple of morphisms
  \[\xymatrix{W \ar[r]^f & X \ar[r]^g & Y \ar[r]^h & Z,}\]
  if $g\cdot f$ and $h \cdot g$ are weak equivalences, then so are $f$, $g$, $h$, and $h \cdot g \cdot f$.
  \item $\cF$ is closed under composition.
  \item Calling a map that is both a weak equivalence and a fibration an {\em acyclic fibration}, all isomorphisms are acyclic fibrations.
  \item $\C$ has a terminal object $\one$.
  \item Pullbacks along fibrations exist; fibrations and acyclic fibrations are stable under pullback.
  \item For any object $X \in \C$, the diagonal morphism $\Delta \colon X \to X \times X$ can be factored as a weak equivalence followed by a fibration:
\[ X \to PX \to X \times X. \]
(Such a factorization, and by abuse of language also the object $PX$, is called a \emph{path object for $X$}.)
  \item Every object is \emph{fibrant}; that is, the unique map $X \to \one$ is a fibration, for any $X \in \C$.
 \end{enumerate}
\end{definition}

\begin{remark}
This is slightly stronger than the original definition given by Brown, in that it requires the class $\cW$ to satisfy the 2-out-of-6 axiom rather than just the more familiar 2-out-of-3. However, once $\C$ satisfies all the other axioms, the following conditions are equivalent (the result is due to Cisinski; see \cite[Thm.~7.2.7]{radulescu-banu:cofibrations}):
 \begin{enumerate}
  \item $\cW$ satisfies 2-out-of-6;
  \item $\cW$ satisfies 2-out-of-3 and is {\em saturated}; that is, if a map $w$ of $\C$ becomes an isomorphism in $\Ho (\C)$, then $w \in \cW$.
 \end{enumerate}
\end{remark}

In this section we show that any categorical model of \FML\ (in the sense of Section~\ref{sec:semantics}) carries the structure of a fibration category; and so, in particular, the syntactic category $\SynCat{\FML}$ does.  From here on, fix some categorical model $\C$ of \FML.

For convenience of exposition, we also assume in this section strong $\eta$-rules for $\Sigma$-types, so that every context is isomorphic to (a context consisting of just) a single iterated $\Sigma$-type: for instance,
\[ [ x \of A,\ y \of B(x) ] \iso [ p : \sum_{x \of A} B(x) ]. \]
This allows us to work just with types, rather than with general contexts.  However, nothing here depends on that assumption; one may simply replace types with contexts and $\Sigma$-types with context extensions, in particular in the definition of the fibrations:

\begin{definition}[Gambino--Garner \cite{gambino-garner}]
A map of $\C$ is a \emph{fibration} if it is isomorphic to some composite of first projections from $\Sigma$-types,
\[\sum_{x \of A} B(x) \rightarrow A.\]
Denote the class of fibrations by $\cF$.
\end{definition}

(This is a slight simplification of Gambino and Garner’s original definition, which also closes $\cF$ under retracts.)  Note that “isomorphic” here refers to the external notion of isomorphism in $\C$, involving definitional equality of maps; and so one cannot represent this definition internally in the type theory, since definitional equality is not represented by a type.  Indeed, there is no way of defining these fibrations internally: every statement of the type theory respects equivalence, and we see in Lemma~\ref{lem:factorization} below that every map is equivalent to a fibration.

Weak equivalences, by contrast, are defined first internally, as in Section~\ref{sec:hott-background} above:
\begin{definition}[Voevodsky]
 \begin{internal}
A map $f \colon A \to B$ is an \emph{equivalence} if for each $b : B$ the homotopy fiber $\hfib (f, b)$ is contractible.

(Note that this is simply a property of $f$, not extra structure, since being an equivalence is a proposition in the sense of Section~\ref{sec:equiv-and-trunc}.)
\end{internal}

Take a map $f \colon A \to B$ in $\C$ to be in $\cW$ if “\textsf{$(\lambda x.\ f(x))$ is an equivalence}” holds in $\C$.
\end{definition}

With these definitions, we are now ready for the main theorem of the section:

\begin{theorem}\label{thm:syn-is-fib-cat}
 $\C$, with $\cW$ and $\cF$ as described above, is a fibration category.
\end{theorem}

We consider the various axioms in turn.

\begin{lemma}
$\cW$ satisfies the 2-out-of-6 property.
\end{lemma}

\begin{proof}
We first show the analogous statement internally (Lemmas \lstinline{two_of_six_hgf}, \lstinline{two_of_six_h}, \lstinline{two_of_six_g}, and \lstinline{two_of_six_f} in the formalization).

\begin{internal}
Let $f$, $g$, $h$ be composable maps, and suppose $f \cdot g$ and $g \cdot h$ are equivalences.  Then:
\begin{itemize}
\item $\inv{(g\cdot f)} \cdot g \cdot \inv{(h \cdot g)}$ gives a quasi-inverse for $h \cdot g \cdot f$;
\item $\inv{(h \cdot g)} \cdot h$ and $f \cdot \inv{(g \cdot f)}$ give left and right inverses for $g$;
\item $\inv{(g \cdot f)} \cdot g$ gives a quasi-inverse for $f$;
\item $g \cdot \inv{(h \cdot g)}$ gives a quasi-inverse for $h$.
\end{itemize}
\end{internal}

This immediately implies the desired external statement, since internal and external composition agree.
\end{proof}

\begin{lemma}\label{lem:pb-of-fib}
Pullbacks of fibrations exist.
\end{lemma}

\begin{proof}
The pullback of a dependent projection is given by substituting into the corresponding dependent type; that is, the following square is a pullback:
\[\xymatrix{
\sum\limits_{x : A'} B(fx) \ar[r] \ar[d] & \sum\limits_{x : A} B(x) \ar[d] \\
A' \ar[r]^f & A.
}\]

The two pullbacks lemma implies that pullbacks of their composites then also exist.
\end{proof}

Note that this is an external statement: these really are strict pullbacks, in contrast to the internally defined pullbacks of Section~\ref{sec:pullbacks}, which from an external point of view become homotopy pullbacks.

\SaveVerb{fiber_to_hfiber_equiv}|fiber_to_hfiber_equiv|
\begin{lemma}[\UseVerb{fiber_to_hfiber_equiv}]\label{lem:fib-ho-fib}
\begin{internal}Let $\pi_1 \colon \sum_{x : A} B(x) \to A$ be a fibration.  Then for any $a : A$, we have $B(a) \equiv \hfib(\pi_1, a)$.\end{internal}
\end{lemma}

\begin{proof} \begin{internal}
Take any $a : A$. For the map $B(a) \to \hfib(\pi_1, a)$, send $b : B(a)$ to $((a,b), \refl(a))$. Conversely, send $((a',b), p) : \hfib(\pi_1, a)$ (where $b : B(a')$ and $p : \Id{a'}{a}$) to the transported element $p_! (b) : B(a)$. The verification that these are mutually inverse is straightforward.
\end{internal} \end{proof}

\begin{lemma}
Fibrations and acyclic fibrations are preserved under pullback.
\end{lemma}

\begin{proof}
Preservation of fibrations is clear by construction from the proof of Lemma~\ref{lem:pb-of-fib}.  For acyclicity, suppose $\pi = \pi_1 \colon \sum_{x : A} B(x) \to A$ is an acyclic fibration, and $f \colon A' \to A$ is a map. Write $f^*\pi$ for the pullback fibration $\pi_1 \colon \sum_{x : A'} B(f(x)) \to A'$.
\begin{internal}Then for any $x:A'$,
\[ \hfib(f^*\pi,x) \equiv B(f(x)) \equiv \hfib(\pi,f(x)) \]
by Lemma~\ref{lem:fib-ho-fib}; and $\hfib(\pi,f(x))$ is contractible by hypothesis, so since equivalence preserves contractibility, $\hfib(f^*\pi,x)$ is again contractible.\end{internal}
So $f^*\pi$ is again acyclic, as required.
\end{proof}

Lemma \lstinline{str_pullback_pres_acyclic_fib} provides the internal part of this argument.


\begin{definition}
The \emph{path type} of a type $A \in \C$ is constructed from its identity types:
 \[ \P A := \sum\limits_{x, y : A} \Id{x}{y} \]
It is equipped by construction with a fibration $\pi$ to $A \times A$, and there is also a natural map $r \colon A \to \P A$ sending $a$ to $((a,a),r(a))$.  Moreover, the map $\P A \to A$ sending $((a,a'),p)$ to $a$ (or to $a'$) gives a quasi-inverse for $r$; so together, we have a factorization of the diagonal of $A$ as a weak equivalence followed by a fibration:
 \[ \xymatrix{ & \P A \ar[rd]^{\pi} & \\
  A \ar[rr]_{\Delta} \ar[ru]^{r} & & A \times A}\]
\end{definition}

We have now amassed all the ingredients of a fibration category:

\begin{proof}[Proof of Theorem~\ref{thm:syn-is-fib-cat}]
Immediate from the preceding lemmas.
\end{proof}

Besides the basic structure, we consider how a few more useful constructions from the theory of fibration categories play out in $\C$:

\begin{lemma}[Factorization Lemma, {\cite[Lem.~11]{gambino-garner}}] \label{lem:factorization}
 For every morphism $f \colon A \to B$ in $\C$, there exists a factorization:
      \[ \xymatrix{ & \P f \ar[rd]^{p_f} & \\
  A \ar[rr]_{f} \ar[ru]^{\sigma_f} & & B}\]
with $\sigma_f \in \cW$ and $p_f \in \cF$.
\end{lemma}

\begin{proof}
 We take
 \[ \P f := \sum_{y : B, x : A} \Id{f x}{y}.\]
 and
 \[ \sigma_f (x) := (f x, x, \refl (f x)).\]
By definition, $p_f$ is in $\cF$; and it is easy to check that $\sigma_f \in \cW$.
\end{proof}

$(\cW,\cF)$ factorizations may be constructed in this way in any fibration category.  In the type-theoretic case, however, they crucially satisfy an additional property, corresponding to the $\synId$-elimination rule: $\sigma_f$ is weakly left-orthogonal to fibrations, and so fibrations form the right class of a weak factorization system.  We will not however go into this point here; see \cite{gambino-garner} for details.

\SaveVerb{right_properness}|right_properness|
\begin{lemma}[\UseVerb{right_properness}] \label{lem:right-proper}
The pullback of a weak equivalence along a fibration is again a weak equivalence:
      \[ \xymatrix{ \pi^*C \ar[r]^-{\pi^* f} \ar[d] & \sum_{x : A} B(x) \ar[d]^{\pi} \\
      A' \ar[r]_-{f \in \cW}  & A}\]
\end{lemma}

\begin{proof} \begin{internal}
The map $\pi^*f$ sends a pair $(y,b)$ to $(f(y), b)$; so taking a quasi-inverse $(g,\eta,\epsilon)$ for $f$, we can construct a quasi-inverse for $\pi^*f$ by sending $(x,b)$ to $(g(x), \pinv{\eta(x)}_! b )$.
\end{internal} \end{proof}

(Again, this lemma holds in any fibration category.)

One may also define cofibrancy, for objects of any fibration category:
\begin{definition}
 An object $C$ of a fibration category $\C$ is \emph{cofibrant} if for any acyclic fibration $p \colon B \to A$ and map $f \colon C \to A$, there is some lifting $\bar{f} \colon C \to B$:
 \[\xymatrix{ & B \ar[d]^p \\
 C \ar@/^/@{-->}[ru]^{\bar{f}} \ar[r]^f & A.}\]
\end{definition}

When $\C$ is a categorical model of \FML\ , we have:
\begin{lemma}
 Every object of $\C$ is cofibrant.
\end{lemma}

\begin{proof}
  Lemma~\ref{lem:fib-ho-fib} implies that every acyclic fibration $\pi_1 \colon \sum_{x:A} B(x) \to A$ admits some section: take some family of contractions of the fibers $\hfib(\pi_1,x)$, and send $x:A$ to the image of the center of contraction $\ast_x : \hfib(\pi_1,x)$ under the equivalence $\hfib(\pi_1,x) \equiv B(x)$.  Now, given $f$ as above, take $\bar{f}$ to be the composite of $f$ with this section.
\end{proof}

We conclude with a somewhat subtler question.  Another condition often assumed for fibration categories is that for any $\N$-indexed sequence
\[ \xymatrix{ A_0 & \ar[l]_{f_0} A_1 & \ar[l]_{f_1} A_2 & \ar[l]_{f_2} \cdots\ ,} \]
if each $f_i$ is a fibration then the sequence has a limit, and moreover the projections from this limit are again fibrations.

This turns out not to be provable in the type theory---in particular, it fails in the syntactic category $\SynCat{\FML}$.  However, appropriate internally-formulated versions of it do hold; this is analogous to the fact that an elementary topos may fail to be externally complete, while possessing all limits in the internal sense.

To see how it fails in $\SynCat{\FML}$, consider the sequence of projections
\[ \xymatrix{ 1 & \ar[l]_{f_0} \synNat & \ar[l]_{f_1} \synNat^2 & \ar[l]_{f_2} \cdots\ } \]
This sequence cannot have a limit, since such a limit would be a $\N$-fold product of copies of $\synNat$, and as such would necessarily have uncountably many global elements, while $\SynCat{\FML}$ is countable.

However, an \emph{internal} limit for the sequence exists, in the form of the object $\synNat^\synNat$ (working internally, it does not make sense to ask whether the projections are fibrations); and, in some models (e.g.\ the simplicial model) this object turns out to be interpreted as the external limit $\prod_{\N} \synNat$.


\section{Limits and applications}
\label{sec:further-development}


\subsection{Pullbacks and equalizers}
\label{sec:pullbacks}

Before defining general limits over graphs, we start by investigating pullbacks; these serve both as a warmup and as a useful tool for subsequent material.

\subsubsection{The standard construction of a pullback}

We start by explicitly constructing the pullback of a cospan. The definitions and theorems in this section are found in \lstinline{Pullbacks}.

\SaveVerb{pullback}|pullback|
\begin{definition}[\UseVerb{pullback}]\label{def:pullback}
Let $\xymatrix{A \ar[r]^{f} & C & B \ar[l]_{g}}$ be a cospan of types and functions. The \emph{(standard) pullback} $\Pb (f, g)$ of this cospan is defined as:
 \[ \Pb(f,g) := \sum_{x : A,\, y : B} \Id{f x}{g y} \]
with the obvious maps:
     \[ \xymatrix{\Pb(f,g) \ar[r]^-{\pi_B} \ar[d]_{\pi_A} & B \ar[d]^{g} \\
      A \ar[r]_{f} & C}.\]
\end{definition}

(This definition may be recast to parallel a traditional construction of the homotopy pullback in fibration categories \cite[Lem~1.3]{brown:abstract-homotopy-theory}: first fibrantly replace $f$ by $p_f$ as in Lemma~\ref{lem:factorization}, obtaining
\[ \xymatrix{
& \P f \mathrlap{ {} = \sum_{c : C} \hfib (f,c) = \sum_{c : C, a : A} \Id{fa}{c}} \ar[d]^{p_f}  \\
            A\ar[r]_{f} & C,
} \qquad \qquad \qquad \qquad \qquad \qquad \]
and then secondly, take the strict pullback of $\P f$ along $g$ as a fibration over $C$, obtaining $\sum_{b : B} \hfib (f,g(c)) = \sum_{b : B, a : A} \Id{fa}{gb}$, which is (strictly, externally) isomorphic to $\Pb(f,g)$ as defined above.)

Note that the pullback is symmetric (\lstinline{pullback_symm}): there is an equivalence $\Pb(f, g) \equiv \Pb (g,f)$ commuting appropriately with the projections and canonical homotopies.

Moreover, the construction of the pullback should be functorial in $(f,g)$. This requires a few extra definitions to state:

\SaveVerb{cospan_map}|cospan_map|
\begin{definition}[\UseVerb{cospan_map}] Given two cospans $\xymatrix{A \ar[r]^{f} & C & B \ar[l]_{g}}$ and $\xymatrix{A' \ar[r]^{f'} & C' & B' \ar[l]_{g'}}$, a \emph{cospan map} $h$ from $(f,g)$ to $(f',g')$ consists of maps $h_A, h_B, h_C$ and homotopies $h_f,h_g$:
\[ \begin{tikzpicture}[scale=1.2]
\node (A) at (0,1) {$A$};
\node (B) at (1.5,2.5) {$B$};
\node (C) at (1.5,1) {$C$};
\node (A') at (1,0) {$A'$};
\node (B') at (2.5,1.5) {$B'$};
\node (C') at (2.5,0) {$C'$};
\draw[cd] (A) to node {$f$} (C);
\draw[cd,swap] (B) to node {$g$} (C);
\draw[cd,swap] (A') to node {$f'$} (C');
\draw[cd] (B') to node {$g'$} (C');
\draw[cd,swap] (A) to node {$h_A$} (A');
\draw[cd] (B) to node[closer] {$h_B$} (B');
\draw[cd,auto=false] (C) to node[closer,descr,pos=0.45] {$h_C$} (C');
\draw[homot,bend left=10,swap] (A') to node[pos=0.6,closer] {$h_f$} (C);
\draw[homot,bend right=10] (B') to node[pos=0.6,closer] {$h_g$} (C);
\end{tikzpicture} \]
\end{definition}

There is an identity map from any cospan to itself (\lstinline{cospan_idmap}); also, there is an evident composition of cospan maps (\lstinline{cospan_comp}).

\SaveVerb{pullback_fmap}|pullback_fmap|
\begin{proposition}[\UseVerb{pullback_fmap}]
A map of cospans $h \colon (f,g) \to (f',g')$ induces a map of pullbacks $\Pb(f,g) \to \Pb(f', g')$. Moreover, this preserves composition and identities.
\end{proposition}

The most frequent application of this functoriality, in practice, is the invariance of pullbacks under equivalences — that, for instance, given a cospan $\xymatrix{A \ar[r]^{f} & C & B \ar[l]_{g}}$ and an equivalence $e \colon A' \equiv A$, there is an equivalence between the pullbacks $\Pb(f,g)$ and $\Pb(f \cdot e, g)$.  This, and various other similar statements, are all easily obtained from the functoriality of $\Pb$ together with the lemma:

\SaveVerb{cospan_equiv_inverse}|cospan_equiv_inverse|
\begin{lemma}[\UseVerb{cospan_equiv_inverse}] 
Suppose $h = (h_A,h_B,h_C,h_f,h_g)$ is a cospan map from $(f,g)$ to $(f',g')$, and $h_A$, $h_B$, $h_C$ are equivalences.  Then there is a cospan map $h^{-1} \colon (f',g') \to (f,g)$, inverse to $h$ in that there are paths $\Id{h \cdot h^{-1}}{1}$ and $\Id{h^{-1} \cdot h}{1}$.
\end{lemma}

An interesting technical point arises here: rather than proving this and other facts about cospan maps directly, we deduce them from the analogous facts about commutative squares (considered as maps between functions).  These are developed in the file \lstinline{CommutativeSquares}. Most immediately, this arrangement slightly simplifies the proofs in the present section, since one does not have to write each construction out separately for the left and right legs of the cospan. It also allows us to directly re-use the commutative squares material in Section~\ref{sec:graph-limits}, as the building blocks of the analogous facts about diagrams over general graphs.

\subsubsection{The universal property of pullbacks}

Above, we defined pullbacks by a specific construction.  Alternatively, one can characterize them by a universal property.  For the next few definitions, fix some cospan $\xymatrix{A \ar[r]^{f} & C & B \ar[l]_{g}}$.

\SaveVerb{cospan_cone}|cospan_cone|
\begin{definition}[\UseVerb{cospan_cone}]\label{def:cospan-cone}
Let $X$ be any type.  A \emph{cone} $\mu$ over $(f,g)$ with vertex $X$ consists of functions $\mu_A$, $\mu_B$, and a homotopy $\mu_C : f \cdot \mu_A \homot g \cdot \mu_B$:
\[ \begin{tikzpicture}
\node (X) at (-1.5,1.5) {$X$};
\node (A) at (-1,0) {$A$};
\node (B) at (0,1) {$B$};
\node (C) at (0,0) {$C$};
\draw[cd,swap,close] (X) to node[pos=0.3] {$\mu_A$} (A);
\draw[cd,close] (X) to node[pos=0.3] {$\mu_B$} (B);
\draw[cd] (A) to node {$f$} (C);
\draw[cd] (B) to node {$g$} (C);
\draw[homot,bend left=20] (A) to node[closer] {$\mu_C$} (B);
\end{tikzpicture} \]
Write $\Cone(X;f,g)$ for the type of cones over $(f,g)$ with vertex $X$.
\end{definition}

$\Cone(X;f,g)$ should be contravariantly functorial in $X$.  We do not show this in full; but in particular, a map $f \colon X' \to X$ induces a map
\[
\Cone(X;f,g) \to \Cone(X';f,g),
\]
given by precomposing the components of the cone with $f$.  For a cone $\mu$, we denote this as $\mu \circ f$.  Fixing a cone $\mu : \Cone(X;f,g)$ thus induces for any type $X'$ a map
\[ (\mu \circ -) \colon (X' \rightarrow X) \to \Cone(X';f,g). \]

This allows us to define the universal property of pullbacks:
\SaveVerb{is_pullback_cone}|is_pullback_cone|
\begin{definition}[\UseVerb{is_pullback_cone}]
A cone $\mu$ over $(f,g)$, with vertex $P$, is an \emph{(abstract) pullback} for $(f,g)$ if for every small type $X : \synU$, the map $(\mu \circ -)$ gives an equivalence $(X \rightarrow P) \equiv \Cone(X;f,g)$.
\end{definition}

One can of course ask whether $(\mu\circ -)$ is an equivalence for an arbitrary type $X$, not necessarily small; but to quantify over types, one must restrict to some universe. Even doing so, the resulting property of “being a pullback” is (a priori) as large as the universe used.  It is, however, a mere proposition, since being an equivalence is one.

(For an investigation of left universal properties of inductive types, defined along similar lines, see \cite{awodey-gambino-sojakova:inductive}.)

\SaveVerb{pullback_universal}|pullback_universal|
\begin{proposition}[\UseVerb{pullback_universal}] \label{prop:pb-up-vs-construction}
The evident cone from the standard pullback $\Pb(f,g)$ (\ref{def:pullback}) to $(f,g)$ is an abstract pullback.
\end{proposition}

\begin{proof}
By direct construction: any cone from some $X$ to $(f,g)$ induces a map $X \to \Pb(f,g)$, and by functional extensionality, this construction is inverse to composition with the standard cone.
\end{proof}

\SaveVerb{abstract_pullback_unique}|abstract_pullback_unique|
\begin{proposition}[\UseVerb{abstract_pullback_unique}]\label{prop:pb-cone-1}
If $\mu : \Cone(X;f,g)$ and $\nu : \Cone(Y;f,g)$ are both pullbacks for $(f,g)$, then the unique map $f \colon Y \to X$ such that $\Id{\mu \circ f}{\nu}$ (provided by the universal property of $\mu$) is an equivalence.

Conversely, if $\mu : \Cone(X;f,g)$ is any cone, and $f \colon X \equiv Y$ an equivalence, then setting $\nu := \mu \circ f$, $\mu$ is a pullback if and only if $\nu$ is.
\end{proposition}

\begin{proof}
The following diagram commutes, and the maps $X \to (1 \rightarrow X)$, $Y \to (1 \rightarrow X)$ are equivalences:
\[ \begin{tikzpicture}
\node (X) at (0,0) {$X$};
\node (X1) at (2,0) {$(1 \rightarrow X)$};
\node (Y) at (0,1.5) {$Y$};
\node (Y1) at (2,1.5) {$(1 \rightarrow Y)$};
\node (C) at (4.5,0.75) {$\Cone(1;f,g)$};
\draw[cd,swap] (Y) to node {$f$} (X);
\draw[cd] (X) to (X1);
\draw[cd] (Y) to (Y1);
\draw[cd,swap,bend right=10] (X1) to node[pos=0.35] {$(\mu \circ -)$} (C);
\draw[cd,bend left=10] (Y1) to node[pos=0.35] {$(\nu \circ -)$} (C);
\end{tikzpicture} \]
It follows by 2-out-of-3 that if any two of $f$, $(\mu \circ -)$, $(\nu \circ -)$ are equivalences, so is the third.
\end{proof}
\todo{The proof in the formalization isn’t currently the same as this!  We should surely change either this or the formalization so they match better?}

\SaveVerb{is_pullback_cone'}|is_pullback_cone'|
\begin{corollary}[\UseVerb{is_pullback_cone'}]
A cone $\mu : \Cone(X;f,g)$ is a pullback cone if and only if the induced map $X \to \Pb(f,g)$ is an equivalence.
\end{corollary}

Since any two interderivable propositions are necessarily equivalent, this property could be used as an alternative definition of $\mu$ being a pullback cone, with the advantage (compared to our previous definition) of yielding again a small type, since it does not quantify over the universe.

\subsubsection{Two pullbacks lemmas}
\label{sec:two-pullbacks}

We have introduced pullbacks in two different ways: via a concrete construction, and via a universal property.  For each of these, one can give a version of the classical two pullbacks lemma.

\SaveVerb{two_pullbacks_equiv}|two_pullbacks_equiv|
\begin{proposition}[\UseVerb{two_pullbacks_equiv}]
\label{prop:conc-two-pullbacks}
For all $f,g,h$ as in the diagram below, the induced comparison map $\Pb(g^* f, h) \to \Pb(f, g \cdot h)$ is an equivalence:
\[ \xymatrix{
\Pb(f, g \cdot h) \ar@/_3ex/[ddr] \ar[dr] \ar@/^3ex/[drrr] & & & \\
&\Pb(g^*f, h) \ar[d] \ar[r] & \Pb(f,g) \ar[d]^{g^* f} \ar[r] & A \ar[d]^f \\
&B_2 \ar[r]^h &  B_1 \ar[r]^g & C
}\]
\end{proposition}

\SaveVerb{abstract_two_pullbacks_lemma}|abstract_two_pullbacks_lemma|
\SaveVerb{Pullbacks3}|Pullbacks3|
\begin{proposition}[\UseVerb{abstract_two_pullbacks_lemma}, in \UseVerb{Pullbacks3}]
\label{prop:abs-two-pullbacks}
 Suppose that in a rectangle
\[ \xymatrix{
P_2 \ar[d] \ar[r] & P_1 \ar[d]^{k} \ar[r] \pb & A \ar[d]^f \\
B_2 \ar[r]^h &  B_1 \ar[r]^g & C
}\]
the right square is a pullback. Then the left square is a pullback if and only if the outer rectangle is a pullback.
\end{proposition}

\begin{proof}
Write $\mu$ for the cone from $P_1$ to $(f,g)$, $\mu'$ for the cone from $P_2$ to $(g^*f,h)$, and $\mu''$ for the cone from $P_2$ to $(f, g \cdot h)$.  Then for any $X$, the following triangle commutes:
\[
\begin{tikzpicture}
\node (X) at (-1,1) {$(X \rightarrow P_2)$};
\node (mu') at (2,2) {$\Cone(X;g^*f,h)$};
\node (mu'') at (2,0) {$\Cone(X;f,g \cdot h)$};
\draw[cd] (X) to node {$(\mu' \cdot -)$} (mu');
\draw[cd,swap] (X) to node {$(\mu'' \circ -)$} (mu'');
\draw[cd] (mu') to (mu'');
\end{tikzpicture}
\]
Here the vertical map denotes the composition of a cone on $(g^*f,h)$ with $\mu$; and this can be shown (by direct construction) to be an equivalence. Hence by 2-out-of-3, $(\mu' \circ -)$ is an equivalence if and only if $(\mu'' \circ -)$ is.
\end{proof}

It should be noted that the arguments involved in showing the equivalence $\Cone(X;g^*f,h) \equiv \Cone(X;f,g \cdot h)$ are necessarily more involved than in the 1-categorical setting, since they depend on comparing paths in types; in terms of the classical theory, this is more analogous to the corresponding lemma for quasi-pullbacks in a bicategory.

\subsubsection{Equalizers}

The formal definitions and theorems corresponding to the remainder of Section~\ref{sec:pullbacks} are found in the file \lstinline{Pullbacks2}, except for the next definition, which appears in \lstinline{Equalizers}.
 
\SaveVerb{equalizer}|equalizer|
\begin{definition}[\UseVerb{equalizer}]
 Let $f, g \colon A \to B$. The \emph{equalizer} of $f$ and $g$ is defined as the type:
 \[ \Eq(f, g) := \sum_{x : A} \Id{fx}{gx}.\]
 together with the projection $\pi \colon \Eq(f,g) \to A$.
\end{definition}

As in classical category theory, pullbacks and equalizers can be defined in terms of each other.

\SaveVerb{eq_as_pb_equiv}|eq_as_pb_equiv|
\begin{proposition}[\UseVerb{eq_as_pb_equiv}]
 The equalizer of any parallel pair $f, g \colon A \to B$ is equivalent to the pullback of the paired map $\langle f ,g \rangle \colon A \to B \times B$ and the diagonal $\Delta_B$:
\[ \phantom{\Eq(f,g) \equiv {}} \xymatrix{
  \mathllap{\Eq(f,g) \equiv {}} \Pb(\Delta_B,\langle f,g \rangle )  \ar[r] \ar[d] & A \ar[d]^{\langle f , g \rangle} \\
  B \ar[r]_{\Delta_B} & B \times B
}\]

Conversely, the pullback of any cospan $\xymatrix{A \ar[r]^{f} & C & B \ar[l]_{g}}$ is equivalent to the equalizer of the pair
 \[ f \cdot \pi_1, g \cdot \pi_2 \colon A \times B \to C.\]
\end{proposition}

\todo{This is a bit of a weak statement.  It would be better to give the fact that a certain specific cones under the equalizer is a pullback, and vice versa.  However we shouldn’t do that here if it’s not in the formalization!}

\subsubsection{Homotopy fibers and loop spaces}

We next consider a couple of examples which bring out the homotopical character of the theory---examples which in classical 1-category theory, and in the type theory with UIP\footnote{“Uniqueness of Identity Proofs”: the axiom that every identity type $\Id[X]{x}{y}$ is a mere proposition \cite{streicher:semantics-book}, \cite{warren:thesis}.}, would be trivial, but which in the un-truncated type theory become non-trivial, corresponding to the classical theory of \emph{homotopy} pullbacks.

We first need one piece of notation.  Given a type $B$ and an element $b : B$, write $\name{b} \colon \synOne \to B$ for the map sending the unique element $* : \synOne$ to $b$.

\SaveVerb{hfiber_to_pullback_equiv}|hfiber_to_pullback_equiv|
\begin{example}[\UseVerb{hfiber_to_pullback_equiv}]\label{ex:hfib-as-pb}
 Given a map $f \colon A \to B$ and an element $b : B$, the homotopy fiber of $f$ over $b$ may equivalently be given as a pullback:
\[ \phantom{\hfib(f,b) \equiv {}}
\xymatrix{ \mathllap{\hfib(f,b) \equiv {}} \Pb(\name{b},f) \ar[r] \ar[d] & A \ar[d]^f \\
              \synOne \ar[r]_{\name{b}} & B}\]
\end{example}

\SaveVerb{Omega_to_pullback_equiv}|Omega_to_pullback_equiv|
\begin{example}[\UseVerb{Omega_to_pullback_equiv}]\label{ex:loop-as-pb}
 Given a type $B$ and an element $b : B$, the space of loops in $B$ based at $b$, $\Omega(B, b) := \Id[B]{b}{b}$ may be given as a pullback:
 \[ \phantom{\Omega(B,b) \equiv {}} \xymatrix{
  \mathllap{\Omega(B, b) \equiv {}} \Pb(\name{b},\name{b}) \ar[r] \ar[d] & \synOne \ar[d]^{\name{b}} \\
  \synOne \ar[r]_{\name{b}} & B.
}\]
\end{example}

This last example in particular exemplifies the fact that our pullbacks correspond, in the classical setting, to \emph{homotopy} pullbacks.

\todo{Again, for both of these, it would be much better to show that a certain specific cone under the hfiber/loop space is a pullback; or equivalently, that the equivalences given above induce the cones one would expect.}

\subsubsection{Properties of pullbacks}

Various nice properties of maps are preserved under pullback.  In proving such preservation properties, the following lemma is rather useful:

\SaveVerb{hfiber_of_pullback}|hfiber_of_pullback|
\begin{proposition}[\UseVerb{hfiber_of_pullback}]
\label{prop:hfiber:pullback}
Given $\xymatrix{A \ar[r]^{f} & C & B \ar[l]_{g}}$, the homotopy fiber of the map $f^*g$ over a point $a:A$ is equivalent to the homotopy fiber of $g$ over $f(a)$.
\end{proposition}

\begin{proof}
Immediate from Example~\ref{ex:hfib-as-pb} together with the concrete two pullbacks lemma, Proposition~\ref{prop:conc-two-pullbacks}.
\end{proof}

\todo{Once again, shouldn’t we also show that this commutes with the hfiber inclusions?}

\SaveVerb{pullback_preserves_equiv}|pullback_preserves_equiv|
\begin{corollary}[\UseVerb{pullback_preserves_equiv}]
Equivalences are stable under pullback.  That is, if $g \colon B \to C$ is an equivalence, then for any $f \colon A \to B$, the pullback $f^* g \colon A \times_B C \to A$ is again an equivalence.
\end{corollary}

\begin{proof}
Each fiber of $f^* g$ is equivalent to some fiber of $g$, so is contractible.
\end{proof}

More generally, any property of maps defined or characterized fiberwise, using an equivalence-invariant property of types, is itself stable under pullback (\lstinline{pullback_preserves_fiberwise_properties}).

\subsection{Limits}
\label{sec:graph-limits}

Generalizing the constructions above of pullbacks and equalizers, we move to limits for diagrams over arbitrary graphs. Unless otherwise noted, the formal definitions and theorems that follow are found in \lstinline{Limits}.

\subsubsection{Graphs and diagrams}

\SaveVerb{graph}|graph|
\begin{definition}[\UseVerb{graph}]
  A \emph{graph} $G$ consists of:
  \begin{itemize}
    \item  a type $G_0$ (the \emph{vertices} or \emph{objects} of $G$); and
    \item  for each $i,j : G_0$, a type  $G_1(i,j)$ (the \emph{edges} or \emph{arrows} from $i$ to $j$).\footnote{Note that we do not assume truncatedness for any of the types involved; we do not need to, essentially since the definiton doesn’t posit any paths within them.  Cf.\ Section~\ref{sec:why-not-cats}.}
  \end{itemize}
\end{definition}

\SaveVerb{diagram}|diagram|
\begin{definition}[\UseVerb{diagram}]
  A \emph{diagram} $D$ on a graph $G$ consists of:
  \begin{itemize}
    \item for each vertex $i \colon G_0$, a type $D_0(i)$;
    \item for each arrow $g : G_1 (i,j)$ of $G$, a map $D_1(g) \colon D(i) \to D(j)$.
  \end{itemize}
\end{definition}

For both graphs and diagrams, we will often suppress the subscripts when they are clear from context.

\SaveVerb{cospan_graph}|cospan_graph|
\SaveVerb{Limits2}|Limits2|
\begin{example}[\UseVerb{cospan_graph}, in \UseVerb{Limits2}]\label{ex:graph-pb}
  To recover cospans as an example of these diagrams, one can define a graph by taking $G_0$ to be the type with three elements, $\{l,m,r\}$ and let $G_1$ be given by:
  \begin{itemize}
    \item $G(l,m) := \synOne$,
    \item $G(r,m) := \synOne$,
    \item $G(i,j) := \emptyset$ otherwise.
  \end{itemize}
 A diagram $D$ over this graph corresponds precisely to a cospan:
 \[\xymatrix{ & D(r) \ar[d] \\
             D(l) \ar[r] & D(m)}\]
\end{example}

\subsubsection{The universal property of limits}

\SaveVerb{graph_cone}|graph_cone|
\begin{definition}[\UseVerb{graph_cone}]\label{def:cone-graph}
  Given a diagram $D$ on a graph $G$, a \emph{cone} $\mu$ on $D$ with vertex $X$ consists of:
  \begin{itemize}
    \item for each $i : G_0$, a map $\mu^0_i \colon X \to D_0(i)$;
    \item for each arrow $g : G_1 (i,j)$, a homotopy $\mu^1_g \colon D_1(g) \cdot \mu^0_i \homot \mu^0_j$.
  \end{itemize}
  Write $\Cone(X;D)$ for the type of cones on $D$ with vertex $X$.
\end{definition}

Again, we usually suppress the subscripts, writing just $\mu_i$, $\mu_f$.

As with cones over cospans, $\Cone(X;D)$ is functorial in $X$: a function $f \colon X' \to X$ and a cone $\mu : \Cone(X;D)$ may be composed to give a cone $\mu \circ f : \Cone(X;D)$.  This lets us generalize the definition of the universal property:

\SaveVerb{is_limit_cone}|is_limit_cone|
\begin{definition}[\UseVerb{is_limit_cone}]\label{def:up-limit}
 Let $D$ be a diagram on the graph $G$. A cone $\mu$ over $D$, with vertex $L$, is an \emph{(abstract) limit} for $D$ if for every small type $X:\synU$, the map $(\mu \circ -) \colon (X \rightarrow L) \to \Cone(X;D)$ is an equivalence.

By abuse of notation, we often speak of $L$ being the limit of $D$, when the cone $\mu$ is implicit.
\end{definition}

Most of the theorems of the preceding section generalize immediately.  In particular,
\begin{proposition}\label{prop:limit-cone-1}
Given any two limit cones for the same diagram, the canonical map between their vertices is an equivalence; conversely, the composition of any limit cone with an equivalence is again a limit cone.
\end{proposition}

Again as in the previous section, there is a standard construction of the limit:
\SaveVerb{limit}|limit|
\begin{definition}[\UseVerb{limit}]
  Let $D$ be a diagram over a graph $G$. The \emph{(standard) limit} $\Lim D$ is the type of pairs $(x,\alpha)$, where
  \begin{itemize}
    \item $x : \prod_{i : G_0} D(i)$;
    \item $\alpha : \prod_{i,j:G_0,\, g : G(i,j)} (\Id{D(g)(x_i)}{x_j})$.
  \end{itemize}
\end{definition}

There is an evident cone from $\Lim D$ to $D$ (\lstinline{limit_graph_cone}), and as one would hope,
\SaveVerb{limit_universal}|limit_universal|
\begin{proposition}[\UseVerb{limit_universal}]
$\Lim(D)$ is an abstract limit for $D$.
\end{proposition}

\SaveVerb{is_limit_cone'}|is_limit_cone'|
\begin{proposition}[\UseVerb{is_limit_cone'}]
A cone $\mu$ from $X$ to some diagram $D$ is a limit for $D$ if and only if the map $X \to \Lim D$ induced by $\mu$ is an equivalence.
\end{proposition}

One again, we may define maps of diagrams (\lstinline{diagram_map}), and show that $\Lim$ is functorial in such maps, and in particular, is functorial in equivalences (\lstinline{limit_fmap_equiv}).  Since graphs, diagrams, and limits are all simply built up from arrows, these definitions and results follow straightforwardly once one has given the basic case of commutative squares, seen as maps between functions. (This is handled in the file \lstinline{CommutativeSquares}.)

\subsubsection{Examples and properties}

\SaveVerb{pb_as_lim_equiv}|pb_as_lim_equiv|
\begin{example}[\UseVerb{pb_as_lim_equiv}, in \UseVerb{Limits2}]
 In Example~\ref{ex:graph-pb} above, we saw that cospans correspond to diagrams over a certain graph.  Then cones over these diagrams correspond to cones over the cospans, as originally defined; and a diagram-cone is a limit exactly if the corresponding cospan-cone is a pullback.
\end{example}

\SaveVerb{lim_as_eq}|lim_as_eq|
\begin{example}[\UseVerb{lim_as_eq}]
Just as in the classical 1-categorical theory, the limit over a diagram $D$ may be constructed as an equalizer of maps between products:
\[ \xymatrix{
  \displaystyle \prod_{\mathclap{i, j : G,\, g : G(i,j)}} \ D(i) \ar@<0.5ex>[r] \ar@<-0.5ex>[r] & \displaystyle \prod_{i : G} D(i)
}\]
\end{example}

Various useful facts are also straightforward to deduce from the standard construction; for instance,
\SaveVerb{trunc_limits_preserve_trunc}|trunc_limits_preserve_trunc|
\begin{proposition}[\UseVerb{trunc_limits_preserve_trunc}, in \UseVerb{Limits2}]
  If $D$ is a diagram on some graph, and each type $D(i)$ is an $n$-type, then $\Lim D$ is an $n$-type; hence via the canonical equivalence, so is any other limit for $D$.
\end{proposition}

\subsubsection{Why not categories?} \label{sec:why-not-cats}

One might reasonably ask here: why have we considered limits only over graphs, not over categories as is usual in the 1-categorical theory?

The problem---as ever in homotopical settings---is one of \emph{coherence}.  Defining a category internally is roughly analogous to defining an $(\infty,1)$-category externally; that is, it requires not only identity, composition, associativity, and the like, but also higher-dimensional data ensuring the coherence of the paths witnessing the associativity axioms, and so on in arbitrarily high dimensions.  While we hope that this will eventually be possible in the type theory, it is currently far from clear how to present it.

In defining categories, this problem can be avoided by assuming truncatedness of the types of morphisms; see \cite{ahrens-kapulkin-shulman} for a development of the resulting theory.  However, to talk about diagrams of arbitrary types over such categories would once again require an infinite family of coherence conditions, essentially since one is presenting an $\infty$-functor into the $(\infty,1)$-category of all types, which is not generally $n$-truncated for any $n$.

However, working with graphs avoids these issues entirely: a map out of a graph (or equivalently, out of the free category thereon) consists purely of 0- and 1-dimensional data, with no coherence required.  (More generally, one could use a similar approach to describe diagrams over finite-dimensional computads or semi-simplicial objects without confronting coherence issues.)

\subsection{Pointed types and fiber sequences}
\label{pointed:section}

\subsubsection{Definitions}

The formal definitions and theorems described in this section are found in \lstinline{PointedTypes}.

\SaveVerb{pointed_type}|pointed_type|
\begin{definition}[\UseVerb{pointed_type}]
 A \emph{pointed type} $(A,a_0)$ is a type $A$, together with an element $a_0 : A$, the \emph{basepoint}.  (We will often refer to both the pointed type and its underlying type as $A$, and write $\pt(A)$ for the basepoint.)
\end{definition}

\SaveVerb{pointed_map}|pointed_map|
\begin{definition}[\UseVerb{pointed_map}]
 A \emph{map of pointed types} (or \emph{pointed map}) $(f,p) \colon (A,a_0) \to (B,b_0)$ consists of a function $f \colon A \to B$, together with a path $p : \Id[B]{f(a_0)}{b_0}$.  (Again, we will often write $f$ for the whole pointed map, and $\pt(f)$ for its associated path.)
\end{definition}

The loop space construction $\Omega$ lifts naturally to a map from pointed types to pointed types, setting $\Omega A := (\Id{\pt\, A}{\pt\, A}, \refl(\pt\, A))$.  One can therefore iterate it, giving the $n$-fold loop spaces $\Omega^n A$ of a pointed type.  Moreover, this has an associated action on maps.  A pointed map $f \colon A \to B $ induces a pointed map $\Omega(f) \colon \Omega A \to \Omega B$, with underlying map sending $q : \Id{\pt\, A}{\pt\, A}$ to $\pinv{\pt\, f} \pdot f[q] \pdot \pt\, f : \Id{\pt\,B}{\pt\,B}$.

Similarly, the homotopy fiber construction $\hfib$ lifts naturally to the point\-ed world.  Given a pointed map $f \colon A \to B$, write $\hfib(f)$ for the pointed type given by $\hfib(f,\pt\, B)$, with basepoint $(\pt\, A, \pt\, f)$; and the inclusion $\hfib(f) \to A$ is again a pointed map.

\subsubsection{The long exact sequence of a pointed map}

As an application of the above tools, we can now recover the long exact sequence associated to a pointed map.  This sequence is a basic but powerful computational tool in classical homotopy theory, and promises to be so also in homotopy type theory: \cite[8.5]{hott:book}, for instance, gives a type-theoretic version of the classical proof that $\pi_3(S^2) \iso \Z$, using the long exact sequence of the Hopf fibration.  Similarly, one can straightforwardly reconstruct the classical theory of covering spaces, as families of \emph{sets} varying over a type, and conclude that they induce isomorphisms of higher homotopy groups.

\SaveVerb{hfiber_ptd}|hfiber_ptd|
\begin{definition}[\UseVerb{hfiber_ptd}]
 A \emph{fiber sequence} consists of a pair $F \to^g E \to^f B$ of pointed maps, together with an equivalence $F \equiv \hfib(f)$ commuting with the inclusion $\hfib(f) \to E$.
\end{definition}

Note that up to canonical equivalence, a fiber sequence is determined simply by the single pointed map $E \to B$.

The following theorem is found in \lstinline{LongExactSequences}.
 
\SaveVerb{hfiber_sequence}|hfiber_sequence|
\SaveVerb{Omega_to_hfiber_seq_0}|Omega_to_hfiber_seq_0|
\begin{theorem}[\UseVerb{hfiber_sequence}, \UseVerb{Omega_to_hfiber_seq_0}, et seq.]\label{thm:les}
 Given a pointed map $f \colon E \to B$, there is a sequence of maps
 \[\ldots \to \Omega^2 B \to \Omega F \to \Omega E \to \Omega B \to F \to E \to B\]
 in which every pair of consecutive maps forms a fiber sequence.
\end{theorem}

\begin{proof}
Taking $F := \hfib(f)$, it is sufficient to prove that the homotopy fiber of the inclusion $F \to E$ is pointed-equivalent to $\Omega B$; subsequent stages follow by iteration. One can prove this equivalence by direct construction; alternatively, the results of Section~\ref{sec:pullbacks} allow us to give a rather more conceptual proof, due originally to Mather \cite[Lem.~32]{mather:pullbacks-in-homotopy-theory}:
 \[\xymatrix{
 \bullet \ar[r] \ar[d] & F \ar[r] \ar[d] & \synOne \ar[d]^{\name{\pt\, B}} \\
 \synOne \ar[r]^{\name{\pt\, E}} & E \ar[r]^f & B
}\]

By the two pullbacks lemma, the pullback of the left-hand square is equivalent to the pullback of the whole rectangle.  But by Examples~\ref{ex:hfib-as-pb} and~\ref{ex:loop-as-pb}, these pullbacks are respectively equivalent to the homotopy fiber of $F \to E$, and to the loop space $\Omega B$.
\end{proof}

\section{Reflections on the formal verification}
\label{reflections:section}

Formalizing the constructions of Sections~\ref{sec:fundamentals} and~\ref{sec:further-development} was often straightforward: many of the definitions are very naturally expressed in the language of type theory, and verifying their properties is often just a matter of unpacking definitions and applying straightforward logical manipulations and background facts.  Sometimes, however, additional effort was required. In this section, we survey some of the practical lessons learned during the formalization.

\subsection{Limitations}
\label{sec:limitations}
One fundamental challenge that arises comes from working purely in the type theory.  In classical approaches to homotopy theory, one always has an extra external scaffolding available, with (in particular) strict, on-the-nose equality on all types of objects.  One typically expects the main results and constructions to respect appropriate notions of equivalence, but one is free to use intermediate constructions that do not.

Developing the homotopy theory in HoTT, we are constrained to work entirely in a homotopy-invariant manner, rendering some classical techniques unavailable.  In most cases, some fully invariant approach is reasonably apparent; but sometimes, one is not.  We saw such a case in Section~\ref{sec:graph-limits}: we do not know how to represent the notion of a diagram over an arbitrary category, and so restricted attention to (diagrams and limits over) graphs.

\subsection{Proof-relevance}
\label{sec:proof-relevance}
Another difficulty lies in getting used to thinking of proofs of equalities as \emph{constructions} that one might need to prove things about later on.

In traditional formalizations, equality is \emph{proof-irrelevant}: different proofs of the same equality are not logically distinguishable.  In Coq, for instance, one could safely end them with the keyword \lstinline{Qed}, which renders them \emph{opaque}, meaning that one cannot later access their contents. In traditional mathematics, this makes sense; once one has an equality, one only needs the fact that it holds, treating the proof as a black box.

In HoTT, however, equality is proof-relevant: a path type may have multiple logically distinct inhabitants.  When constructing equality proofs in this setting, one typically needs to end an equality proof with the keyword \lstinline{Defined}, allowing the user to unfold that definition later on.  The specifics of the proof matter; one tries to keep proofs as clean and short as possible, using lemmas and constructions with known, previously proven properties.  Unfortunately, this means that several of Coq’s powerful tactics (notably the \lstinline{rewrite} family) are somewhat unsatisfactory in our setting: the paths they produce are difficult to reason about later.

On the other hand, some important statements remain proof-irrelevant.  If a type has been shown to be a proposition, one knows that any two elements of it are canonically equal; so one may make such an element opaque without losing any logical content.  Even so, it is often convenient to leave such objects transparent, to retain their \emph{computational} content.

For instance, for a function \lstinline{f}, the type \lstinline{IsEquiv f} (the property that \lstinline{f} is an equivalence) is a proposition; so in principle one may safely render a proof of this opaque.  However, one often uses such a proof to produce an inverse for \lstinline{f}; if the proof was transparent, then the resulting inverse will retain computational properties from its construction, whereas if the proof is opaque, one must reason explicitly about the action of the inverse.  We formed no clear convention on this: sometimes it turned out more convenient to keep such proofs transparent, for easier reduction in later proofs; in other case, this was unnecessary, and making the proofs opaque gave more efficient compilation.

\subsection{Constructing paths}
\label{sec:constructing-paths}
The most fundamental type constructor in homotopy type theory is the type of paths, and the most challenging parts of proofs usually involved constructing paths between complex objects. Given the subject matter, we never had to pass beyond the 2-categorical level, constructing paths between paths; but even so, this required a good deal of care, and facility with path algebra.

One recurring situation was the construction of paths between elements of a dependent sum, or elements of a record type with dependencies between components. For example, if $( a, b )$ and $( a', b' )$ are elements of a type $\sum_{x : A} B(x)$, constructing a path between these two elements involves constructing a path $p$ from $a$ to $a'$, and then constructing a path $q$ from the transport of $b$ along $p$ to $b'$. Thus in general we have:
\begin{lstlisting}
  Lemma total_paths {A : Type} {B : A -> Type}
    {s s' : total B}
    (p : paths (pr1 s) (pr1 s'))
    (q : paths (p # (pr2 s)) (pr2 s'))
  : s = s'.
\end{lstlisting}
where \lstinline{pr1} and \lstinline{pr2} denote the projections from the total space $\sum_{x : A} B(x)$. For interactive, tactic-based proofs, we generally found it useful to bundle the arguments $p, q$ into a single structure:
\begin{lstlisting}
  Lemma total_paths' {A : Type} {B : X -> Type}
    {s s' : total B}
  : { p : pr1 s = pr1 s' & p # pr2 s = pr2 s' } -> s = s'.
\end{lstlisting}
Recall that here \lstinline${ p : pr1 s = pr1 s' & p # pr2 s = pr2 s' }$ is notation for a dependent sum, denoting the type of pairs $(p, q)$ as above. When constructing a path between elements of a dependent sum, even when $p$ is explicitly available, applying \lstinline{(total_paths p)} sometimes fails to infer implicit arguments. Instead, applying \lstinline{total_paths'} leaves the goal of providing the pair $(p, q)$, providing the user explicitly with their required types. The tactic \lstinline{exists p} can then be used to give the first component, leaving the goal of constructing the second path $q$ interactively.

The problem is that \lstinline{transport} is rather difficult to work with. There are many library lemmas about how its behaviour depends on the dependent type $B$, which in principle allow one to work with transported terms; but we found it more convenient to directly give tailored variants of \lstinline{total_paths} for each specific $\Sigma$- and record type.

For example, taking a cospan $f \colon A \to C$, $g \colon B \to C$, the standard pullback of $f$ and $g$ is given by the type $\sum_{x : A, y : B} \Id{f x}{g y}$.  Using \lstinline{total_path} to provide a path in this type between triples \lstinline{(x;(y;p))}, \lstinline{(x';(y';p'))} would require three paths \lstinline{q : x = x'}, \lstinline{r : q # y = y'}, and \lstinline{s : r # q # p = p'}.  Notice, however, that in this case the second component, $y$, does not depend on $x$, so the transport is trivial; and moreover, the doubly-transported third component can be explicitly described as a composite. Thus, one can provide the following lemma to construct a path between two elements of the standard pullback:
\begin{lstlisting}
  Definition pullback_path' {A B C : Type} {f : A -> C} {g : B -> C}
    (u u' : pullback f g)
  : { p : pullback_pr1 u = pullback_pr1 u'
      & {q : pullback_pr2 u = pullback_pr2 u'
      & (ap f p)^ @ (pullback_comm u) @ (ap g q)
        = pullback_comm u' } }
    -> u = u'.
\end{lstlisting}
The process of analyzing the canonical data for presenting a path between elements of a complex type, and writing lemmas to construct and work with such paths, was crucial to the formalization.

To consider one last example of this sort, recall that a cospan cone, that is, a diagram on the data $f, g$ above, consists of a space, $X$, and maps $h$ and $k$ from $X$ to $A$ and $B$, respectively, making the diagram commute.
\begin{lstlisting}
  Definition cospan_cone {A B C : Type} (f : A -> C)
    (g : B -> C) (X : Type)
  := { h : (X -> A) &
       { k : (X -> B) & forall x, (f(h x)) = (g(k x)) }}.
\end{lstlisting}
A path between two such cones involves, in particular, a path between the family of paths in the third component:
\begin{lstlisting}
  Definition cospan_cone_path
    {A B C : Type} {f : A -> C} {g : B -> C} {X : Type}
    {Phi1 Phi2 : cospan_cone f g X}
    (p : cospan_cone_map1 Phi1 = cospan_cone_map1 Phi2)
    (q : cospan_cone_map2 Phi1 = cospan_cone_map2 Phi2)
    (r : forall x:X,
      cospan_cone_comm Phi1 x = (ap f (ap10 p x)) @ 
          cospan_cone_comm Phi2 x @ (ap g (ap10 q x))^)
  : Phi1 = Phi2.
\end{lstlisting}
Here, \lstinline{cospan_cone_map1}, \lstinline{cospan_cone_map2}, and \lstinline{cospan_cone_comm} refer to the three components of a cospan cone in the preceding definition. As with \lstinline$total_paths$, we also give a version \lstinline$cospan_cone_path'$ that packages the required components into a dependent sum, and is often more convenient in interactive proofs.

The advantage to these formulations is that it is comparatively straightforward (using lemmas from the HoTT library) to reason about transport operations their interactions with each other, as well as with path operations such as concatenation and inversion.

Returning to the question of the path-algebra itself, we found the formalization to require significant facility with such calculations. The HoTT library has a number of tactics for automating common manipulations and simplifications, but we found these tactics generally slowed down the proof-checker significantly.  So, for the most part, we ended up giving such calculations by hand, building them explicitly from basic lemmas.

\subsection{General strategies}

We found it important to develop our theories and proofs in a modular way. The value of modularity in interactive theorem proving is well understood (see, for example, \cite{gonthier:feit-thompson-final}), but in the context of homotopy type theory, it takes on additional significance. For one thing, many statements involving paths can only be proved when stated in full generality (to make available the elimination for $\synId$-types).  As a consequence, some facts cannot be derived in the course of a proof, on the fly, but have to be expressed independently.  The fact that one often needs to reason about the construction of paths provides an additional reason to construct such proofs out of individually-named component lemmas: doing so allows one derive properties of the components individually, and then invoke these properties later on.  In other words, reasoning about a modularly-constructed proof allows one to work with the individual lemmas and unpack their contents selectively, as needed. In contrast, the failure to modularize can result in formal terms that are overwhelming in complexity.

Perhaps the most important lesson we learned was not to expect too much from an interactive theorem prover. Although homotopy type theory provides a powerful framework to support homotopy-theoretic reasoning, one still needs a thorough understanding of the relevant mathematics. To get some of the more complex proofs and constructions to work, we found it vitally important to find the right definitions, the right way of formulating assertions, the right supporting infrastructure, and the right proof strategies.  This required thinking carefully about the mathematical content, avoiding the temptation to simply dive in and hack.

This should not suggest that Coq was no help at all. Indeed, Coq was excellent for helping us keep track of definitions and formulate statements correctly.  Especially for more complex path-constructions, applying standard rules to unwrap and reduce the contents of a goal type was an extremely useful aid to finding the term required.  In practice, we found ourselves going back and forth between the blackboard and Coq, using Coq to negotiate the inevitable syntactic bureaucracy, and then returning to the blackboard to recoup intuitions and plan proof-strategies. In this way, Coq earned its keep, serving as a ``proof assistant'' in a very real sense.

\subsection{A case study: the two pullbacks lemma}
\label{sec:two-pullbacks-discussion}

We close with a discussion of the abstract two pullbacks lemma, Proposition~\ref{prop:abs-two-pullbacks}, by way of illustration. Somewhat to our surprise, this turned out to be the most difficult proof in our formalization.  In the end, we tried three substantially different approaches before finding one satisfactory.  All three can be found in \lstinline{Pullbacks3_alt}.

Consider for now just the forward direction of Proposition~\ref{prop:abs-two-pullbacks}, which states that if both squares have the universal property of pullbacks, then so does the composite. Let $f \colon A \to C$, $g \colon B_1 \to C$, $h \colon B_2 \to C$, and $k \colon P_1 \to B_1$ denote the maps so labeled in the diagram there.  Our first approach invoked the concrete two pullbacks lemma, Proposition~\ref{prop:conc-two-pullbacks}, which states that
\[
 \Pb(g^* f, h) \equiv \Pb(f, g \cdot h).
\]
We then derived the following chain of equivalences, using the fact that cones from $X$ to the cospan $(f, g)$ are equivalent to maps from $X$ to the standard pullback:
\begin{eqnarray*}
 (X \rightarrow P_2) & \equiv & \Cone(X; k, h)  \\
  & \equiv & (X \rightarrow \Pb(k, h)) \\
  & \equiv & (X \rightarrow \Pb(g^* f, h)) \\
  & \equiv & (X \rightarrow \Pb(f, g \cdot h)) \\
  & \equiv & \Cone(X; f, g \cdot h).
\end{eqnarray*}
Here the second and last equivalences are just the universal properties of the concrete pullbacks.  The notation $g^* f$ in the third equivalence denotes the pullback of $f$ along $g$ according to the concrete pullback construction; this equivalence relies on the fact that any abstract pullback is equivalent to the concrete one, and the fact that the concrete pullback construction is functorial.  The fourth equivalence is just (post-composition with) the concrete two pullbacks equivalence, Proposition~\ref{prop:conc-two-pullbacks}.

The equivalence of the left- and right-hand sides of the chain above \emph{almost} gives what we want: however, the universal property for the outer pullback square requires not just that an equivalence exists, but that the \emph{canonical} map from $X \to P_2$ to $\Cone(X; f, g \cdot h)$ is an equivalence.

What remains is thus to show that the map we have just constructed is homotopic to the canonical one!  This, however, turned out to be extremely difficult. The problem was a failure of modularity: all we could do was unwrap the long, complicated term, and calculate. We managed to do this, but although the tactic engine declared the effort successful, we were unable to get it past the type-checker (presumably because the resulting term was too large).

Our second approach involved constructing the desired inverse by hand. Any cone $\mu : \Cone(X; f, g \cdot h)$ over the outer cospan can be reinterpreted as a cone $\mu' : \Cone(X; f, g)$ over the right cospan. Applying the universal property of the cone from $P_1$, we obtain a map $m_1 \colon X \to P_1$ inducing $\mu'$; we can then take $m_1$ as the first leg of a cone $\mu'' : \Cone(X; k, h)$ on the left cospan. Applying the universal property of the cone from $P_2$ then gives a map $m_2 \colon X \to P_2$, as desired.  However, the task of proving that this construction is indeed a two-sided inverse for $(\mu \circ -)$ turned out to be difficult. For example, the first task requires one to show that, starting with a cone $\mu : \Cone(X; f, g \cdot h)$, carrying out the procedure above to obtain a map from $X$ to $P_2$ and then taking the induced cone, the resulting cone $\nu : \Cone(X; f, g)$ is connected by a path to the original $\mu$. As described in Section~\ref{sec:constructing-paths}, this involves showing not only that the component maps agree, but also that the resulting families of equality proofs agree as well; this turns out to be an interesting but laborious exercise in bicategorical path-algebra.

We finally settled on the approach described in Section~\ref{sec:two-pullbacks}, which establishes both directions of Proposition~\ref{prop:abs-two-pullbacks} simultaneously. Showing that the type $\Cone(X; k, h)$ of cones on the left cospan is equivalent to the type $\Cone(X; f, g \cdot h)$ of cones on the outer cospan required some effort, but the result was still considerably cleaner than either of the previous proofs.  With that in hand, all that remained was to show that the triangle depicted in the proof of Proposition~\ref{prop:abs-two-pullbacks} in Section~\ref{sec:two-pullbacks} commutes. To our very pleasant surprise, this fact had a one-line proof in Coq:
\begin{lstlisting}
  Lemma two_pullback_triangle_commutes {P1 : Type}
      (C1 : cospan_cone f g P1)
      {P2 : Type} (C2 : cospan_cone (cospan_cone_map2 C1) h P2)
      {X : Type} (m : X -> P2)
  : top_cospan_cone_to_composite C1 (map_to_cospan_cone C2 X m)
  = map_to_cospan_cone (top_cospan_cone_to_composite C1 C2) X m.
  Proof.
    exact 1.
  Defined.
\end{lstlisting}
In other words, the left- and right-hand sides are definitionally equal.

\bibliographystyle{amsalphaurlmod}
\bibliography{fibcats-bib}

\end{document}